\documentclass[12pt, reqno]{amsart}
\usepackage{amsmath, amsthm, amscd, amsfonts, amssymb}
\usepackage[bookmarksnumbered]{hyperref}

\newtheorem{theorem}{Theorem}[section]

\theoremstyle{definition}

\newtheorem{example}[theorem]{Example}

\theoremstyle{remark}

\numberwithin{equation}{section}

\begin{document}

\title[]{On approximate n-ring homomorphisms and n-ring derivations}
\author[M. Eshaghi Gordji]{M. Eshaghi Gordji }
\address{Department of Mathematics, Semnan University, P. O. Box 35195-363, Semnan, Iran}
\email{madjid.eshaghi@gmail.com}



\begin{abstract}
Let $A,B$ be two rings and let $ X$ be an $ A-$module. An additive
map $h: A\to B$ is called n-ring homomorphism if
$h(\Pi^n_{i=1}a_i)=\Pi^n_{i=1}h(a_i),$ for all $a_1,a_2, \cdots,a_n
\in {A}$. An additive map $D: A\to  X$ is called $n$-ring derivation
if
$$D(\Pi^n_{i=1}a_i)=D(a_1)a_2\cdots a_n+a_1D(a_2)a_3\cdots a_n+\cdots +a_1a_2\cdots a_{n-1}D(a_n),$$
for all $a_1,a_2, \cdots,a_n \in {\mathcal A}$.  In  this paper we
investigate the Hyers-Ulam-Rassias stability of $n$-ring
homomorphisms and n-ring derivations.
\end{abstract}
\maketitle


\section{Introduction and preliminaries}

Let $A,B$ be two rings (algebras). An additive (linear) map $h: A\to
B$ is called n-ring homomorphism (n-homomorphism) if
$h(\Pi^n_{i=1}a_i)=\Pi^n_{i=1}h(a_i),$ for all $a_1,a_2, \cdots,a_n
\in {A}$. The concept of n-homomorphisms was studied for complex
algebras by Hejazian, Mirzavaziri, and Moslehian [12] (see also [7],
[9], [10], [22]).

Let $A$ be ring and let $ X$ be an $ A-$module. An additive map $D:
A\to  X$ is called $n$-ring derivation if
$$D(\Pi^n_{i=1}a_i)=D(a_1)a_2\cdots a_n+a_1D(a_2)a_3\cdots a_n+\cdots +a_1a_2\cdots a_{n-1}D(a_n),$$
for all $a_1,a_2, \cdots,a_n \in { A}$. A 2-ring derivation is then
a ring derivation, in the usual sense, from an algebra into its
module. Furthermore, every ring derivation is clearly also a n-ring
derivation for all $n\geq 2,$ but the converse is false, in general.
For instance let
\[{\mathcal A} := \left[ \begin{array}{cccc}
{0} & {\Bbb R} & {\Bbb R} & {\Bbb R}\\
{0} & {0} & {\Bbb R} & {\Bbb R}\\
{0} & {0} & {0} & {\Bbb R}\\
{0} & {0} & {0} & {0}\\
 \end{array} \right] \]
then $\mathcal A$ is an  algebra equipped with the usual matrix-like
operations. It is easy to see that $$\mathcal A^3\neq 0=\mathcal
A^4.$$ Then every additive map $f:\mathcal A \to \mathcal A$ is a
4-ring derivation.

We say that a functional equation (*) is stable if any function $f$
approximately satisfying the equation (*) is near to an exact
solution of (*).  Such a problem was formulated by S. M.~Ulam [26]
in 1940 and solved in the next year for the Cauchy functional
equation by D. H.~Hyers [13] in the framework of Banach spaces.
Later, T. Aoki [2] and Th. M. Rassias [25] considered mappings $f$
from a normed space into a Banach space such that the norm of the
Cauchy difference $f(x+y)-f(x)-f(y)$ is bounded by the expression
$\epsilon(\|x\|^p+\|y\|^p)$ for all $x,y$ and some $\epsilon\geq 0$
and  $p\in [0,1)$. The terminology "Hyers-Ulam-Rassias stability"
was indeed originated from Th. M. Rassias's paper [25] (see also
[8], [23], [15], [18]).

D. G.~Bourgin is the first mathematician dealing with the stability
of ring homomorphisms. The topic of approximate ring homomorphisms
was studied by a number of mathematicians, see [3, 5, 14, 6, 16, 20,
23, 24] and references therein.

It seems that approximate derivations was first investigated by
K.-W. Jun and D.-W. Park [17]. Recently, the stability of
derivations have been investigated by some authors; see [1, 4, 11,
17, 19, 21] and references therein. In  this paper we investigate
the Hyers-Ulam-Rassias stability of $n$-ring homomorphisms  and
n-ring derivations.

\section{Main result}

We start our work with a result  concerning approximate n-ring
homomorphisms, which can be regarded as an extension of Theorem 1 of
[3].

\begin{theorem}
Let  $ A$ be a ring, $ B$ be a Banach algebra and let $\delta$ and
$\varepsilon$ be nonnegative real numbers. Suppose $f$ is a mapping
from $ A$ to $ B$ such that
$$\|f(a+b)-f(a)-f(b)\|\leq\varepsilon \eqno (2.1)$$ and that
$$\|f(\Pi^n_{i=1}a_i)-\Pi^n_{i=1}f(a_i)\|\leq\delta \eqno (2.2)$$
for all $a, b, a_1, a_2,...,a_n  \in  A$. Then there exists a unique
n-ring homomorphism $h:  A \to  B$ such that
$$\|f(a)-h(a)\|\leq\varepsilon \eqno(2.3)$$
for all $a \in  A$. Furthermore,
$$(\Pi^k_{i=1}h(a_i))(\Pi^n_{i=k+1}f(a_i)-\Pi^n_{i=k+1}h(a_i))=(\Pi^k_{i=1}f(a_i)-\Pi^k_{i=1}h(a_i))(\Pi^n_{i=k+1}h(a_i))=0~~~
\eqno (2.4)
$$
 for all $a_1,a_2,...,a_n\in  A$ and all $k\in \{1,2,...,n-1\}$.
\end{theorem}
\begin{proof}
Put $h(a)=\lim_m\frac{1}{2^m}f(2^ma)$ for all $a\in A.$ Then by
Hyers' Theorem, $h$ is additive. We will show that $h$ is n-ring
homomorphism. For every $a_1,a_2,...,a_n\in  A$ we have
\begin{align*}
h(a_1a_2...a_n)&=\lim_m \frac{1}{2^m}f(2^m(a_1a_2...a_n))=\lim_m
\frac{1}{2^m}f((2^ma_1)a_2...a_n)\\
&=\lim_m
\frac{1}{2^m}\{f((2^ma_1)a_2...a_n)-f(2^ma_1)(\Pi^n_{i=2}f(a_i))+f(2^ma_1)(\Pi^n_{i=2}f(a_i))\}\\
&\leq \lim
_m\frac{1}{2^m}\{\delta+f(2^ma_1)(\Pi^n_{i=2}f(a_i))\} \hspace {0.4 cm} (by~~~~~~ (2.2))\\
&=h(a_1)(\Pi^n_{i=2}f(a_i)).
\end{align*}
Hence
$$h(a_1a_2...a_n)=h(a_1)(\Pi^n_{i=2}f(a_i)).\eqno (2.5)$$
By (2.5) it follows that
$$h(a_1)f(2^ma_2)(\Pi^n_{i=3}f(a_i))=h(2^ma_1a_2...a_n)=2^mh(a_1a_2...a_n)$$
for all $a_1,a_2,...,a_n\in  A, m\in \Bbb N$. Dividing both sides of
above equality by $2^m$ and taking the limit $m\rightarrow\infty $.
Then we have
$$h(a_1)h(a_2)(\Pi^n_{i=3}f(a_i))=\lim_m h(a_1)\frac{1}{2^m}f(2^ma_2)(\Pi^n_{i=3}f(a_i))=h(a_1a_2...a_n)$$
Hence by (2.5)  we have

$$h(a_1)h(a_2)(\Pi^n_{i=3}f(a_i))=h(a_1a_2...a_n)=h(a_1)(\Pi^n_{i=2}f(a_i)).$$
Now, proceed in this way to prove that
$$(\Pi^k_{i=1}h(a_i))(\Pi^n_{i={k+1}}f(a_i))=h(a_1a_2...a_n)\eqno (2.6)$$
for all $a_1,a_2,...,a_n\in  A$ and all $k\in \{1,2,...,n-1\}$. Put
 $k=n-1$ in (2.6), we obtain
$$(\Pi^{n-1}_{i=1}h(a_i))f(2^ma_n)=h(2^m(a_1a_2...a_n))=2^mh(a_1a_2...a_n)\eqno (2.7)$$
for all $a_1,a_2,...,a_n\in A, m\in \Bbb N.$ Dividing both sides of
(2.7) by $2^m$ and taking the limit  $m\rightarrow\infty $, it
follows that $h$ is a n-homomorphism. On the other hand $h$ is
additive and $h(a)=\lim_m\frac{1}{2^m}f(2^ma)$ for all $a\in A.$
Then we have
$$(\Pi^k_{i=1}f(a_i))(\Pi^n_{i={k+1}}h(a_i))=h(a_1a_2...a_n)=(\Pi^n_{i={1}}h(a_i))\eqno (2.8)$$
for all $a_1,a_2,...,a_n\in  A$ and all $k\in \{1,2,...,n-1\}$, and
(2.4) follows (2.6) and (2.8). Obviously the uniqueness property of
$h$ follows from additivity.
\end{proof}

Similarly to the proof of Theorem 2 of [3], we can prove the
Hyers-Ulam-Rassias type stability of n-ring homomorphisms as
follows.

\begin{theorem}
Let  $ A$ be a normed algebra,  $ B$ be a Banach algebra, $\delta$
and $\varepsilon$ be nonnegative real numbers and let $p,q$ be a
real numbers such that $p,q<1$ or $p,q>1$. Assume that $f:A\to B$
satisfies the system of functional inequalities
$$\|f(a+b)-f(a)-f(b)\|\leq\varepsilon (\|a\|^p+\|b\|^p)$$ and
$$\|f(\Pi^n_{i=1}a_i)-\Pi^n_{i=1}f(a_i)\|\leq\delta (\Pi^n_{i=1}\|a_i\|^q)$$
for all $a, b, a_1, a_2,...,a_n  \in  A$. Then there exists a unique
n-ring homomorphism $h:  A \to  B$ and a constant $k$ such that
$$\|f(a)-h(a)\|\leq k\varepsilon \|x\|^p $$
for all $a \in  A$.
\end{theorem}

Now we will  prove the stability of n-ring derivations from a normed
algebra into a Banach module.

\begin{theorem} Let ${\mathcal A}$ be a normed algebra and let $\mathcal X$ be a Banach $\mathcal A-$ module. Suppose the map
 $f:{\mathcal A}\longrightarrow{\mathcal X}$ satisfying the system
 of inequalities:

$$\|f(a+b)-f(a)-f(b)\|\leq \epsilon (\|a\|^p+\|b\|^p) \hspace {0.5 cm}(a,b\in \mathcal A),\eqno \hspace {0.5 cm}(2.9)$$

$$\|f(\Pi^n_{i=1}a_i)-f(a_1)\Pi^n_{i=2}a_i-a_1f(a_2)\Pi^k_{i=3}a_i-\cdots -\Pi^{n-1}_{i=1}a_if(a_n)\|\leq
 \epsilon (\sum^n_{i=1}\|a_i\|^p)  \eqno(2.10)$$
for all $a_1,a_2,...,a_n\in \mathcal A,$ where $\epsilon$ and p are
constants in $\Bbb R^+ \cup\{0\}.$ If $p\neq 1,$ then there is a
unique n-ring derivation $D:{\mathcal A}\longrightarrow{\mathcal X}$
such that
$$\|f(a)-D(a)\|\leq \frac{2\epsilon}{2-2^p}\|a\|^p \eqno (2.11)$$ for all $a\in \mathcal
A.$ Moreover if for every $c\in \Bbb C$ and $a\in \mathcal A,$
$f(ca)=cf(a),$ then $f=D.$
\end{theorem}
\begin{proof}
By Rassias's Theorem and (2.9), it follows that there exists a
unique additive mapping $D:{\mathcal A}\longrightarrow{\mathcal A}$
satisfies  (2.11). We have to show that $D$ is a n-derivation. Let
$s=\frac{1-p}{|1-p|}$, and let $a,a_1,a_2,...,a_n\in \mathcal A.$
For each $m\in \Bbb N$, we have $D(a)=m^{-s}D(m^sa),$ therefore
\begin{align*}
\|m^{-s}f(m^sa)-D(a)\|&=m^{-s}\|f(m^sa)-D(m^sa)\|\\&\leq
m^{-s}\frac{2\epsilon}{2-2^p}\|a\|^p \|m^sa\|^p\\
&= m^{s(p-1)}\frac{2\epsilon}{2-2^p}\|a\|^p.
\end{align*}
Since $s(1-p)\leq 0$ then we have
$$\lim_m \|m^{-s}f(m^sa)-D(a)\|=0. \eqno (2.12)$$
 Similarly we can show that
$$ \|m^{-ns}f(m^{ns}\Pi^n_{i=1}a_i)-D(\Pi^n_{i=1}a_i)\|\leq
 m^{ns(p-1)}\frac{2\epsilon}{2-2^p}\|\Pi^n_{i=1}a_i\|^p.$$
 Therefore we have

$$\lim_m \|m^{-ns}f(m^{ns}\Pi^n_{i=1}a_i)-D(\Pi^n_{i=1}a_i)\|^p=0. \eqno (2.13)$$
By (2.10), for each $m\in \Bbb N$ we have
\begin{align*}
\|m^{-ns}f(m^{ns}\Pi^n_{i=1}a_i)&-m^{-s}f(m^{s}a_1)\Pi^n_{i=2}(a_i)-m^{-s}a_1f(m^{s}a_2)\Pi^n_{i=3}(a_i)\\
&-...- m^{-s}\Pi^{n-1}_{i=1}(a_i)f(m^{s}a_n)\|\\
&=m^{-ns}\|f(\Pi^n_{i=1}(m^{s}a_i))-f(m^sa_1)(\Pi^n_{i=2}(m^{s}a_i))\\
&-\sum^{n-1}_{j=2}\Pi^{j-1}_{l=1}(m^{s}a_l)f(m^sa_j))\Pi^{n}_{l=j+1}(m^{s}a_l)-\Pi^{n-1}_{l=1}(m^{s}a_l)f(m^sa_n)\|\\
&\leq m^{-ns}\epsilon \Pi^n_{i=1}\|m^{s}a_i\|^p=m^{ns(p-1)}\epsilon
\Pi^n_{i=1}\|a_i\|^p.
\end{align*}
Thus we have
\begin{align*}
\lim_m\|m^{-ns}f(m^{ns}\Pi^n_{i=1}a_i)&-m^{-s}f(m^{s}a_1)\Pi^n_{i=2}(a_i)-m^{-s}a_1f(m^{s}a_2)\Pi^n_{i=3}(a_i)\\
&-...- m^{-s}\Pi^{n-1}_{i=1}(a_i)f(m^{s}a_n)\|=0 \hspace {3.8 cm}
(2.14)
\end{align*}
for all $a_1,a_2,...,a_n\in \mathcal A$. On the other hand we have
\begin{align*}
\|D(\Pi^n_{i=1}a_i)&-D(a_1)\Pi^n_{i=2}a_i-a_1D(a_2)\Pi^n_{i=3}a_i-\cdots
-\Pi^{n-1}_{i=1}a_iD(a_n)\|\\
&  \leq  \|D(\Pi^n_{i=1}a_i)-m^{-ns}f(m^{ns}\Pi^n_{i=1}a_i)\|+\|m^{-ns}f(m^{ns}\Pi^n_{i=1}a_i)\\
&-m^{-s}f(m^{s}a_1)\Pi^n_{i=2}(a_i)-m^{-s}a_1f(m^{s}a_2)\Pi^n_{i=3}(a_i)\\
&-...- m^{-s}\Pi^{n-1}_{i=1}(a_i)f(m^{s}a_n)\|\\
&+\|m^{-s}f(m^sa_1)\Pi^n_{i=2}a_i-D(a_1)\Pi^n_{i=2}a_i\|\\
&+\|m^{-s}a_1f(m^sa_2)\Pi^n_{i=3}a_i-a_1D(a_2)\Pi^n_{i=3}a_i\|\\
&+...\\
&+\|m^{-s}\Pi^{n-1}_{i=1}a_i f(m^sa_n)-\Pi^{n-1}_{i=1}a_iD(a_n)\|
\end{align*}
for all $a_1,a_2,...,a_n\in \mathcal A$. According to (2.10), (2.13)
and (2.14), if  $m\rightarrow \infty$, then the right hand side of
above inequality tends to 0, so we have
$$D(\Pi^n_{i=1}a_i)=D(a_1)\Pi^n_{i=2}a_i+a_1D(a_2)\Pi^n_{i=3}a_i+\cdots
+\Pi^{n-1}_{i=1}a_iD(a_n),$$ for all $a_1,a_2,...,a_n\in \mathcal
A$. Hence  $D$ is a n-ring derivation.  The uniqueness property of
$D$ follows from additivity. Let now for every $c\in \Bbb C$ and
$a\in \mathcal A,$ $f(ca)=cf(a),$ then by (2.11), we have
\begin{align*}\|f(a)-D(a)\|&=\|m^{-s}f(m^sa)-m^{-s}D(m^sa)\|\\
&\leq m^{-s}\frac{2\epsilon}{2-2^p}\|m^sa\|^p\\
&= m^{s(p-1)}\frac{2\epsilon}{2-2^p}\|a\|^p \end{align*} for all
$a\in \mathcal A.$ Hence by letting $m\rightarrow \infty$ in above
inequality, we conclude that $f(a)=D(a)$ for all $a\in \mathcal A$.

\end{proof}
Similarly we can prove the following Theorem which can be regarded
as an extension of Theorem 2.6 of [11].
\begin{theorem} Let $p$, $q$ be real numbers such that $p,q<1$, or $p,q>1$.
Let ${\mathcal A}$ be a Banach algebra and let $\mathcal X$ be a
Banach $\mathcal A-$ module. Suppose the map
 $f:{\mathcal A}\longrightarrow{\mathcal X}$ satisfying the system
 of inequalities:

$$\|f(a+b)-f(a)-f(b)\|\leq \epsilon (\|a\|^p+\|b\|^p) \hspace {0.5 cm}(a,b\in \mathcal A),$$

$$\|f(\Pi^n_{i=1}a_i)-f(a_1)\Pi^n_{i=2}a_i-a_1f(a_2)\Pi^k_{i=3}a_i-\cdots -\Pi^{n-1}_{i=1}a_if(a_n)\|\leq
 \epsilon (\pi^n_{i=1}\|a_i\|^q)  $$
for all $a_1,a_2,...,a_n\in \mathcal A,$ where $\epsilon$ and p are
constants in $\Bbb R^+ \cup\{0\}.$ Then there exists  a unique
n-ring derivation $D:{\mathcal A}\longrightarrow{\mathcal X}$ such
that
$$\|f(a)-D(a)\|\leq \frac{2\epsilon}{2-2^p}\|a\|^p \eqno (2.11)$$ for all $a\in \mathcal
A.$
\end{theorem}

The following counterexample, which is a modification of Luminet's
example (see [16]), shows that Theorem 2.2 is failed for $p=1$ (see
[3]).

\begin{example}
Define a function $\varphi:{\mathbb R}\to {\mathbb R}$ by
$$\varphi(x)=\begin{cases}
0 & |x|\leq 1\\
&\qquad\qquad.\\
x\ln(|x|) & |x|>1
\end{cases}$$
Let $f:\Bbb R\to {\mathcal M_3(\Bbb R)}$ is defined by

\[f(x)=\left[ \begin{array}{ccc}
{0} & 0 & 0 \\
{\varphi(x)} & {0} & 0 \\
{0} & {0} & {0} \\
\end{array} \right]\]
for all $x\in {\mathbb R}$. Then
$$\|f(a+b)-f(a)-f(b)\|\leq \varepsilon(\|a\|+\|b\|)$$
and
$$|f(\Pi^n_{i=1}a_i)-(\Pi^n_{i=1}f(a_i))|\leq \delta(\Pi^n_{i=1}\|a_i\|^2)$$ for some $\delta>0, \varepsilon>0$ and all
$a_1,a_2,..a_n\in \Bbb R $; see [3]. Therefore $f$ satisfies the
conditions of Theorem 2.2 with $p=1$, $q=2$. There is however no
n-ring homomorphism  $h:\Bbb R\to {\mathcal M_3 (\Bbb R)}$ and no
constant $k>0$ such that
$$\|f(a)-h(a)\|\leq k\varepsilon\|a\| \qquad (a\in \Bbb R)\,.$$
\end{example}
Also example 2.7 of [11] shows that Theorem 2.4 above is failed for
$p=1$.

\end{document}